\documentclass{amsproc}

\usepackage{amssymb}

\usepackage{}

\newtheorem{theorem}{Theorem}[section]
\newtheorem{proposition}[theorem]{Proposition}
\newtheorem{lemma}[theorem]{Lemma}
\newtheorem{corollary}[theorem]{Corollary}

\theoremstyle{definition}
\newtheorem{definition}[theorem]{Definition}

\theoremstyle{remark}

\numberwithin{equation}{section}

\numberwithin{equation}{section}
\newcommand{\be}{\begin{equation}}
\newcommand{\ee}{\end{equation}}

\newcommand{\bbC}{{\mathbb C}}

\newcommand{\bbR}{{\mathbb R}}

\newcommand{\bbN}{{\mathbb N}}

\newcommand{\calD}{{\mathcal D}}

\newcommand{\calL}{{\mathcal L}}

\newcommand{\calH}{{\mathcal H}}

\newcommand{\calR}{{\mathcal R}}

\newcommand{\calS}{{\mathcal S}}

\newcommand{\calO}{{\mathcal O}}

\newcommand{\inner}[2]{\langle#1,#2\rangle}

\newcommand{\norm}[1]{\lVert#1\rVert}

\newcommand{\PB}[2]{\{  #1\,,\,#2 \}}

\newcommand{\wt}{\widetilde}

\usepackage{color}

\begin{document}

% \title[short text for running head]{full title}
\title[pseudospectra of Schr\"odinger operators on Zoll manifolds]{On the pseudospectra of Schr\"odinger operators on Zoll manifolds}

%    Only \author and \address are required; other information is
%    optional.  Remove any unused author tags.

%    author one information
% \author[short version for running head]{name for top of paper}

\author{D. Sher}
\address{Mathematics Department\\ DePaul University\\Chicago, IL 60604}
\curraddr{}
\email{dsher@depaul.edu}
\thanks{D. Sher supported in part by NSF grant EMSW21-RTG 1045119.}

%    author two information
\author{A. Uribe}
\address{Mathematics Department\\
		University of Michigan\\Ann Arbor, Michigan 48109}
\curraddr{}
\email{uribe@umich.edu}
\thanks{}

%    author three information
\author{C. Villegas-Blas}
\address{Instituto de Matem\'aticas, UNAM, Unidad Cuernavaca}
\curraddr{}
\email{villegas@matcuer.unam.mx}
\thanks{C. Villegas-Blas partially  supported by projects PAPIIT-UNAM   IN105718
	and CONACYT 000000000283531}

%    The 2010 edition of the Mathematics Subject Classification is
%    the current definitive version.

\date{}

\begin{abstract}
	We consider non-self-adjoint
	Schr\"odinger operators $\Delta +V$ where 
	$\Delta$ is the Laplace-Beltrami operator on a Zoll manifold $X$
	and $V\in C^\infty(X,\bbC)$.  We obtain asymptotic results on the pseudo-spectrum
	and numerical range of such operators.
\end{abstract}

\maketitle

%    Text of article.

%    Bibliographies can be prepared with BibTeX using amsplain,
%    amsalpha, or (for "historical" overviews) natbib style.

%\hfill\today %REMOVE THIS PRIOR TO SUBMISSION

\section{Introduction and statement of results}
Let $\Delta$ denote the Laplace-Beltrami operator on a manifold $X$ and let $V\in C^{\infty}(X,\mathbb C)$.  We will study asymptotic properties of the pseudospectrum and numerical range of the Schr\"odinger operator $\Delta +M_V$, where $M_V$ is the operator of multiplication by $V$. Here we are taking $\Delta$ as a self-adjoint operator with a suitable domain in $L^2(X)$.

Recall that, given an operator $P$ densely defined on some Hilbert space $\calH$, and
$\epsilon>0$, the $\epsilon$-pseudospectrum of $P$ is the set of complex numbers $\lambda$
such that
\[
\norm{(P-\lambda I)^{-1}}\geq\frac{1}{\epsilon},
\]
where, by definition, the norm on the left-hand side is infinite if $\lambda$ is an eigenvalue of $P$.
Equivalently, $\lambda$ is in the $\epsilon$-pseudospectrum of $P$ iff there exists $\psi\in\calH\setminus\{0\}$
such that
\[
\frac{\norm{(P-\lambda I)\psi}}{\norm{\psi}}\leq \epsilon.
\]
This motivates the following definition:

\begin{definition} We say that a sequence of complex numbers $\{\lambda_k\}$ is in the \emph{asymptotic pseudospectrum} of $\Delta+M_V$ if there exists a sequence $\{\psi_k\}$ of functions in $ L^2(X)$ such that as $k\to\infty$,
	\begin{equation}\label{lapapa}
	\frac{\norm{(\Delta+M_V-\lambda_k I)\psi_k}}{\norm{\psi_k}}= O(k^{-\infty})\quad\text{and}\quad
	\lim_{k\to\infty}|\lambda_k| = +\infty.
	\end{equation}
	The sequence $\{\psi_k\}$ is referred to as an associated \emph{pseudo mode}.
\end{definition}
Note that (\ref{lapapa}) 
implies that, $\forall \epsilon >0$, $\lambda_k$ will be in the $\epsilon$-pseudospectrum
of $\Delta +M_V$ for all sufficiently large $k$.
The previous definition has an analogue in the semi-classical setting, see for example \cite{Zw} \S 12.5 (and references therein).
In the present case
the role of Planck's constant is played by $\hbar = 1/k$.

It is well-known that the spectrum of a non-self-adjoint operator is very unstable. 
 The pseudospectrum, however, is much more stable, which makes it a natural object to study. It is also known that points in
the pseudospectrum do not have to be close to the spectrum, even in the semi-classical limit.

\medskip

In this paper we consider the case where $X$ is a closed, connected Zoll manifold.  More specifically, 
we make the assumption there is a $T>0$ such that each geodesic on $X$ has minimal period $T$ (we rule out the existence of exceptional short geodesics).
Such manifolds are often called $C_l$-manifolds \cite{bes78}.  The simplest examples are
the spheres with the round metric.   It is known that
the spectrum of the Laplacian $\Delta$ on a Zoll manifold consists of clusters 
of eigenvalues of uniformly bounded width centered at the points 
\[
\Lambda_k=\frac{4\pi^2}{T^2}\left(k+\frac{\beta}{4}\right)^2,\quad k=1,2,\ldots ,
\]
where $\beta$ is the Morse index of one periodic geodesic 
(hence all periodic geodesics).  Specifically
\begin{equation}\label{cumulos2}
\exists \; C>0\ \text{such that}\qquad \sigma(\Delta) \subset\bigcup_{k\in\bbN}
[\Lambda_k -C\,,\, \Lambda_k+C],
\end{equation}
where $\sigma(\Delta)$ denotes the spectrum of $\Delta$.  
Note that the distance between the 
clusters $\Lambda_k$ and $\Lambda_{k+1}$ is linear in $k$, which allows us to separate the eigenvalues 
of the Laplacian unambiguously into clusters, at least for large $k$.

%For each $k$, let 
%\[
%\Pi_k: L^2(X)\to E_k
%\] 
%be the projection onto the direct sum, $E_k$, 
%of eigenspaces of $\Delta$ corresponding to the $k$-th cluster, and let 
%$d_k$ be the dimension of $E_k$. Colin de Verdi\`ere showed in \cite[Thm 1.4]
%{cdv79} that in the case of a $C_l$-manifold, $d_k$ is an explicit polynomial in $k$.  (In the more general 
%Zoll setting, there could be oscillatory terms as $k$ goes to infinity.) 

%To state our results we first have to establish the general structure of the spectrum of $\Delta+M_V$. 
%We begin by recalling the structure of the Laplacian on Zoll manifolds.
More precisely, as proved 
in \cite{cdv79}, there exist self-adjoint, commuting pseudo-differential operators $A$ of order one and 
$Q_0$ of order zero such that the spectrum of $A$ consists of the eigenvalues $k=1,2,\ldots$ and
%$ Q_{-1}$ 
%of order $-1$, both of which commute with ${\Delta}$, such that $\sigma(A)\subset\mathbb N$ and
%\begin{equation}\label{structureZoll}
%\sqrt{\Delta}=\frac{2\pi}{T}\left(A+\frac{\beta}{4}\right)+ Q_{-1}.
%\end{equation}
%Note that $A$ commutes with $ Q_{-1}$ as well.   If we let
%$ Q_0=\frac{4\pi}{T}A Q_{-1}$,  we obtain that
\begin{equation}\label{structureZoll2}
\Delta =\frac{4\pi^2}{T^2}\left(A+\frac{\beta}{4}\right)^2+ Q_0. %+ R_{-1},
\end{equation}
%where $ R_{-1}\in\Psi^{-1}(X)$, and, for all $k$,
In particular, we can take  $C$ equals to the operator norm  $\|Q_0\|$ with $Q_0\in\mathcal{B}(L^2(X))$.  
Moreover, if for each $k$ we let
\begin{equation}\label{}
E_k := \text{eigenspace of } A\ \text{corresponding to the eigenvalue } k,
\end{equation}
the eigenvalues in the $k$-th cluster consist of the eigenvalues
of the restriction $Q_0|_{E_k}$, shifted by $\Lambda_k$.

The operator $(A+\frac{\beta}{4})^2$ has spectrum contained in $(\mathbb N+\frac{\beta}{4})^2$, 
and has the same principal and sub principal symbols as $\Delta$ up to the factor $\frac{4\pi^2}{T^2}$. 
For example, in case $X=S^n$, the unit $n$-dimensional sphere, $T=2\pi$ and the eigenvalues
of the Laplacian are $k(k+n-1) = (k+\frac{n-1}{2})^2 - \frac{(n-1)^2}{4}$.  Therefore in this case
\[
\beta = 2(n-1)\quad\text{and}\quad Q_0= - \frac{(n-1)^2}{4}I.
\] 
For the round sphere the principal symbol of $Q_0$ is constant.
In the general Zoll case  the symbol of $ Q_0$ is not constant
(see \cite{zel97}, Theorem 3, for an expression for it), but it is always constant along
geodesics since $[Q_0, \Delta] =0$.  The fact  that the symbol if $Q_0$ is not constant corresponds to 
the non-trivial clustering phenomenon of the Laplace eigenvalues
seen in Zoll metrics, as the asymptotic behavior of the eigenvalue clusters  
is given to first order by the principal symbol of $ Q_0$ via a Szeg\"o limit theorem.

In conclusion, our Schr\"odinger operator $\Delta + M_V$ has the form
\begin{equation}\label{conclusionSchrodinger}
\Delta+M_V=\frac{4\pi^2}{T^2}\left(A+\frac{\beta}{4}\right)^2+ Q_0+M_V.%)+ R_{-1}.
\end{equation}
Thus we can think of $M_V$ as a non-self-adjoint perturbation added to 
$Q_0$.

\subsection{Statements of our results}

We can now state the main results of this paper.

\subsubsection{On the spectrum}
First we will show that 
the spectrum of $\Delta+M_V$ also has a cluster structure (at least if we are sufficiently far from the origin) generalizing (\ref{cumulos2}) (see also \cite{GK69}). Let us denote by $D$  the closed disk around the origin with radius $||Q_0 + M_V||$. 

\begin{theorem}\label{Cumulos}
	Let $V: X\to\bbC$ be a bounded measurable function.  Then the spectrum of $H:= \Delta +M_V$
	is contained in the union
	\begin{equation}\label{cumulos}
	\bigcup_{k=0}^\infty D_k,\quad D_k := \Lambda_k+ D.
	\end{equation}
	  The spectrum of $H$ consists 
	  entirely of isolated eigenvalues with finite algebraic multiplicities and, for $k$ sufficiently large, 
the spectrum of  $H$ in the the disk $D_k$ consists of eigenvalues with total algebraic multiplicity equal to the 
dimension of $E_k$.
%total geometric multiplicity of the eigenvalues of $\Delta$ in the interval  $[  \Lambda_k - ||Q_0||,\Lambda_k + ||Q_0|| ]$. 
Moreover, the generalized eigenspaces of $\Delta+M_V$  span $L^2(X)$.
\end{theorem}
Note that the sets ${D}_k$ are pairwise disjoint for large $k$.

\subsubsection{Results on the asymptotic pseudospectrum}
%Having discussed the general structure of the spectrum of $\Delta + M_V$, 
We now turn to the
statement of our results on the asymptotic pseudospectrum of $\Delta+M_V$.

Let $H_0(x,\xi) = \norm{\xi}$, $H_0\in C^\infty(T^*X\setminus\{0\})$.
This Hamiltonian generates geodesic flow re-parametrized by arc
length.  In what follows we shall think of oriented geodesics as trajectories of the
Hamiltonian $H_0$ on the unit tangent bundle of $X$.
Let $\calO$ be the manifold of oriented geodesics on $X$.  This is the
symplectic manifold obtained by symplectic
reduction of $T^*X\setminus\{0\}$ under the circle action generated by $H_0$.  
Given any smooth function $V$ on $X$, one can define its Radon transform to be the function $\wt V$ on $\calO$ defined by
\[
\wt V(\gamma) = \frac{1}{T}\int_\gamma V\, ds,
\]
where $s$ denotes arc length along the oriented geodesic $\gamma\in\calO$, 
and we are
abusing the notation and continue to denote by $V$ the pull-back of the potential
to the unit tangent bundle.  
Note that $\calO$ has 
an involution corresponding to reversing the orientation of the geodesics.  This involution is
anti-symplectic, and $\wt V$ is even with respect to it.

Now recall that $Q_0$ commutes with $\Delta$ (and with $(A+\frac{\beta}{4})$). 
Therefore its principal symbol $q_0$ is invariant along geodesics and may thus be viewed as a 
function on $\calO$, also denoted $q_0$.  We reviewed above that, in case $X$ is the unit sphere, 
$q_0=-\frac{(n-1)^2}{4}$ (constant).
With this notation, our  result is as follows:

\begin{theorem}\label{Main}
	(1) Let $\mu = (\wt{V}+q_0)(\gamma)$, where $\gamma\in\calO$ is such that
	\[
	\PB{\Re\wt{V}+q_0}{\Im\wt{V}}(\gamma) < 0.
	\]
	Then $\{\Lambda_k+\mu\}$ is in the asymptotic pseudospectrum of $\Delta+M_V$.  In fact, for each $m\in\gamma$,
	there exists a pseudo mode $\{\psi_k\}$ such that the sum
	$\sum_{k=1}^\infty \psi_k$ is a distribution whose wave-front set is the ray in $T^*X\setminus\{0\}$
	through $m$.
	
	(2) If $\{\Lambda_k+\mu\}$  is in the asymptotic pseudospectrum with $\mu = (\wt{V}+q_0)(\gamma)$, and
	\[
	\PB{\Re\wt{V}+q_0}{\Im\wt{V}}(\gamma) >0,
	\] then for any associated pseudo mode $\{\psi_k\}$,
	the wave-front set of $\sum_k \psi_k$ is disjoint from the conic set in $T^*X\setminus\{0\}$
	generated by $\gamma$.
\end{theorem}

Note that although the potential $\wt V$ is even under the involution 
of reversing the orientation of $\gamma$, the Poisson bracket 
$\PB{\Re\wt{V}+q_0}{\Im\wt{V}}(\gamma)$ is odd under the same involution.
Together with part (1) of the theorem, this implies that the set of sequences of the form
\[
\left\{\lambda_k+\wt{V}(\gamma)\;;\; \PB{\Re\wt{V}+q_0}{\Im\wt{V}}(\gamma)\not=0\right\}
\]
is contained in the asymptotic pseudospectrum.
However, in general this set does not equal the asymptotic pseudospectrum.  For example, in the case of spheres,
if $V$ is odd (with respect to the antipodal map) then $\wt V$ is identically zero.
In that case, however, there is a substitute function for $\wt V$, namely the function
\begin{equation}\label{}
\wt{\wt{V}}=\frac 14 \wt{V^2}- \frac{1}{8\pi} \int_0^{2\pi}\, dt\int_0^t\PB{\phi_t^*V}{\phi_s^*V}\,ds
\end{equation}
where $\phi_t$ is geodesic flow, see \cite{AU85} (Theorem 3.2).  More subtly we expect
that one can have that $ \PB{\Re\wt{V}+q_0}{\Im\wt{V}}(\gamma)=0$
and have pseudo modes concentrated at $\gamma$ with pseudo eigenvalue $\wt V(\gamma)$,
in the spirit of Theorem 2.1 in \cite{P-S04}, where non-vanishing of higher-order derivatives
is assumed.

Theorem 1.3 is analogous to Theorem 12.8 in \cite{Zw} (originally due to Davies), where a similar Poisson
bracket condition appears.  It is also analogous in spirit to the main result in \cite{BU03}.

\subsubsection{Results on the numerical range}
Recall that the numerical range of an operator $\hat{H}$ on $L^2(X)$ with domain $\calD(\hat{H})$ is
the subset of the complex plane defined by:
\[
\calR(\hat{H}) = \left\{\frac{\inner{\hat{H}(\psi)}{\psi}}{\inner{\psi}{\psi}}\;;\;\psi\in\calD(\hat{H})\right\}.
\]
We are interested in the numerical range of $\hat{H} = \Delta + M_V$.  We will only consider
the asymptotics of the sets  of values
\[
\left\{\frac{\inner{\hat{H}(\psi)}{\psi}}{\inner{\psi}{\psi}}\;;\;\psi\in E_k\right\}.
\]
To state our result we shift these sets by $\Lambda_k$:
\begin{theorem}\label{NumRangeThm} In case $X=S^n$ or $X$ a Zoll surface, 
	the limit $\calR_{\infty}$ of the sets
	\[
	\calR_k = \left\{ \frac{\inner{(\Delta + M_V)\psi}{\psi}}{\inner{\psi}{\psi}}-\Lambda_k
	\;;\;\psi\in E_k\setminus\{0\}
	\right\}
	\]
	as $k$ tends to infinity is the convex hull of the image of $\wt{V}+q_0:\calO\to\bbC$.
\end{theorem}
For a similar result for Berezin-Toeplitz operators see \cite{BU03}.

\section{Proof of theorem \ref{Cumulos}}

Let $D$ be the operator $\frac{4\pi^2}{T^2}\left(A+\frac{\beta}{4}\right)^2$.   
Let $z$ be in the resolvent set $\rho(D)$ of the operator $D$, with the distance
$d(z,\sigma(D))$ from $z$ to the spectrum $\sigma(D)$ of $D$  greater than the norm
$|| Q_0 + M_V ||$.  Since $D$ is self adjoint then    $|| \left(D-z\right)^{-1} || = 1/d(z,\sigma(D))$ which implies 
 $|| \left( Q_0 + M_V \right) \left(D-z\right)^{-1}  ||< 1$. Therefore the operator  $\left(I + (Q_0+M_V) \left(D-z\right)^{-1} \right)^{-1}$ is  a bounded   operator which in turn implies that 
 $(\Delta + M_V - z)^{-1}$ exists and it is a bounded operator, i.e.  $z$ is in the resolvent set $\rho(H)$  of $H$.  
Thus the spectrum $\sigma(H)$ of H  must be contained in the union of the closed disks $D_k$ indicated in (\ref{cumulos}).  

Consider $z\in\bbC$ such that $d(z,\sigma(\Delta)) > 2\|M_V\|$.   Since $\Delta$   has compact resolvent then writing 
$H-z =  \left( I + M_V(\Delta-z)^{-1} 	\right) (\Delta-z)$ we see that $(H-z)^{-1}$ exists and it must be compact (note that since $\Delta$ is self adjoint then  $\|(\Delta-z)^{-1}\|=1/d(z,\sigma(\Delta))$).   Thus we have that the spectrum of $H$ consists entirely of eigenvalues with finite algebraic multiplicities, see reference \cite{Ka}, theorem III, section 6.8.

Let $g(k)=\Lambda_k-\Lambda_{k-1}=O(k)$ and let ${\mathcal C}_k$ be the circle
of radius $g(k)/2$ and center $\Lambda_k$.  Note that, for k sufficiently large, both the interval of radius $||Q_0||$ around $\Lambda_k$ and the disk 
$D_k$ are inside ${\mathcal C}_k$.  
   Moreover,   
\[
\forall z\in{\mathcal C}_k,\quad d(z,\sigma(\Delta))\geq g(k)/3 
\textrm{ and } d(z,\sigma(\hat{H}))\geq g(k)/3.
\]
In particular, ${\mathcal C}_k$ is contained in both $\rho(\Delta)$ and $\rho({H})$.
Thus the following two projections are well defined:
\begin{equation}\label{proyectorPk}
P_k=\frac{1}{2\pi\imath}\oint_{{\mathcal C}_k}\left(z-{H}\right)^{-1}dz
\end{equation}
and
\begin{equation}\label{proyectorPik}
\Pi_k=\frac{1}{2\pi\imath}\oint_{{\mathcal C}_k}\left(z-\Delta\right)^{-1}dz.
\end{equation}
Since $\Delta$ is self adjoint then  $\|(z-\Delta)^{-1}\|= 1/d(z,\sigma(\Delta)) = O(k^{-1})$ uniformly in  ${\mathcal C}_k$.  This implies for k sufficiently large that 
$ \|  M_V(z-\Delta)^{-1}\| \leq 1/2$ which in turn implies, using a Neumann series, that
\begin{eqnarray}
\| \left(z-{H}\right)^{-1} \| &\leq& \|  (z-\Delta)^{-1} \| \; \| \left(  I + M_V(z-\Delta)^{-1} \right)^{-1}   \| \\
&\leq& \frac{1}{d(z,\sigma(\Delta))} \; \frac{1}{1- \| M_V   (z-\Delta)^{-1} \|}  = O(k^{-1})
\end{eqnarray}
uniformly in  ${\mathcal C}_k$. 

Thus we obtain
\begin{eqnarray}\label{normadif}
\|P_k-\Pi_k\|&=&\left\|\frac{1}{2\pi\imath}\oint_{{\mathcal C}_k}\left(z-{H}\right)^{-1}M_V\left(z-\Delta\right)^{-1}dz\right\| \nonumber \\
&\leq&\frac{1}{2\pi}\oint_{{\mathcal C}_k}\|\left(z-{H}\right)^{-1}\| \; \|M_V\|  \; \|\left(z-\Delta\right)^{-1}\||dz| \nonumber\\
&\leq&\frac{1}{2\pi}(2\pi g(k))O(k^{-1})  \; ||M_V|| \; O(k^{-1})=O(k^{-1}).
\end{eqnarray}
Therefore, for $k$ sufficiently large, the norm  $\|P_k-\Pi_k\|$
is less than one which in turn implies  that the dimension of
the range of $P_k$,  $  \text{dim}(\text{Ran}(P_k)) $, and the dimension of the range of $\Pi_k$, 
$\text{dim}(\text{Ran}(\Pi_k)) $, must
be the same, see reference  \cite{Ka}, chapter I, section 4.6.   Moreover,  $\text{dim}(\text{Ran}(P_k)) $ must be finite which implies that the spectrum of $H$ inside ${\mathcal C}_k$ consists only of eigenvalues, see reference \cite{Ka}, chapter III, section 6.5.   We conclude that the  total algebraic multiplicity of the eigenvalues of $H$ inside $D_k$ must be the same as the initial geometric multiplicity of the cluster of eigenvalues of $\Delta$ in the interval 
$[\Lambda_k-\|Q_0\|,\Lambda_k+\|Q_0\|]$, namely the dimension of $E_k$. 

Finally, for $1<p<\infty$, let us denote by ${\mathcal S}_p$  the space of all completely continuous operators  $A$ for which $\sum_{j=1}^{\infty}s_j^p(A)<\infty$ where $s_j(A) $ are the eigenvalues of $|A|\equiv\sqrt{A^*A}$.   
The fact that the generalized eigenspaces of $\Delta+M_V$  span $L^2(X)$ is a consequence of theorem 10.1 in reference \cite{GK69} noting that 
$\Delta$ has discrete spectrum and for all $\lambda$  in the resolvent set of $\Delta$,   $(\Delta-\lambda)^{-1}V(\Delta-\lambda)^{-1}$ belongs to a Schatten class ${\mathcal S}_p$ for some  
$1<p<\infty$. This last fact is a consequence of (i) the property that ${\mathcal S}_p$ is an ideal in ${\mathcal B}(L^2(X))$, (ii) $V(\Delta-\lambda)^{-1}$ is a bounded operator and (iii)   $(\Delta-\lambda)^{-1}$ belongs to  ${\mathcal S}_p$ for some $1<p<\infty$  (the multiplicity of the eigenvalues of $\Delta$ in the interval $[\Lambda_k -C\,,\, \Lambda_k+C]$ is given by a polynomial in $k$ for $k$ sufficiently large, see reference \cite  {cdv79}).

\section{Proof of theorem \ref{Main}}

\subsection{The averaging method}
We begin by recalling the ``averaging method" (\cite{wei77, Gu81}):
\begin{proposition}
	If $V:X\to\bbC$ is smooth, there exists a pseudo-differential operator of order $(-1)$, $F$, such that
	\[
	e^F\bigl(\Delta + M_V\bigr)e^{-F}=e^{F}\bigl(\frac{4\pi^2}{T^2}(A+\frac{\beta}4)^2+ Q_0+M_V\bigr)e^{-F} = \frac{4\pi^2}{T^2}(A+\frac{\beta}{4})^2 + B+S,
	\]
	where $B$ is a pseudo-differential operator of order zero such that $[A,B] = 0$, and $S$ is a smoothing operator.
	Moreover, the
	symbol of $B$ is $\wt V+q_0$, regarded as a function on $T^*X^n\setminus\{0\}$.
\end{proposition}
\begin{proof}  We'll sketch the proof, for completeness. Let 
	\[
	\wt A = \frac{2\pi}{T}(A+\frac{\beta}{4});
	\]
	then $\wt A^2+Q_0=\Delta$, and $\wt A$ takes the value $\sqrt{\Lambda_k}$ on the space $E_k$.
	We proceed inductively, with $F_1$ an operator or order $(-1)$ solving the equation
	\newcommand{\ave}{\text{\tiny ave}}
	\[
	[F_1, \wt A^2] = B_0 - (M_V+ Q_0) + \text{order }(-1),
	\]
	where $B_0$ is the operator
	\[
	B_0 = \frac{1}{T}\int_0^{T} e^{it\wt A}(M_V+ Q_0)e^{-it\wt A}\, dt.
	\]
	It is easy to see that $B_0$ commutes with $\wt A$ and
	hence with $A$.
	
	For $F_1$ we take an operator with principal symbol
	\[
	f_1:= -\frac{H_0^{-1}}{2T}\int_0^{T}dt \int_0^t\, \phi_s^*( V+q_0)\,ds
	\]
	where $H_0(x,\xi) = \norm{\xi}$ and $\phi_s$ is its Hamilton flow. Note that the principal symbol of $\wt A^2$ is the same as the principal symbol of $\Delta$, i.e. $H_0^2$.
	The symbol of $[F_1,\wt A^2]$ is
	\[
	\PB{f_1}{H_0^2} = 2H_0\PB{f_1}{H_0} =
	-\frac{1}{T}\int_0^{T}\bigl(\phi_t^*(V+q_0) - (V+q_0)\bigr)\,dt = V-V^\ave,
	\]
	where $V^\ave+q_0 = \frac{1}{T}\int_0^{T} \phi_t^*(V)\,dt+q_0$ is the principal symbol of $B_0$. Note that the factors of $q_0$ cancel since $\phi_t^*(q_0)=q_0$ (i.e. $q_0$ is already averaged over the flow).
	It follows that
	\[
	e^{F_1}\bigl(\wt A^2+ Q_0 + M_V\bigr)e^{-F_1} = \wt A^2+ B_0 + R_{-1},
	\]
	where $R_{-1}$ is an operator of order $(-1)$.
	
	Using the same constructions, one can
	modify $F_1$ by an operator $G_1$ of order $(-2)$ so that, if
	$F_2 = F_1 + G_1$,
	\[
	e^{F_2}\bigl(\wt A^2+ Q_0 + M_V\bigr)e^{-F_2} = \wt A^2 + B_0 + B_{-1} + R_{-2}
	\]
	where $R_{-2}$ is now an operator of order $(-2)$ and both $B_0$ and $B_{-1}$ commute with $A$.
	This process can be continued indefinitely, and after a Borel summation we can find an
	operator $F_\infty$, with the same principal symbol as $F_1$, such that
	\[
	e^{F_\infty}\bigl(\wt A^2+ Q_0 + M_V\bigr)e^{-F_\infty} = \wt A^2 + B +S,
	\]
	where $[A, B ]= 0$ and $S$ is a smoothing operator, and where the principal symbol of $B$ is the same as that of $B_0$.
\end{proof}

\subsection{Constructing the pseudo mode}

We now proceed to construct a pseudo mode for $\Delta + M_V$, under the hypotheses
of  part (1) of Theorem \ref{Main}.
We will do this by first constructing a pseudo mode for $\wt A^2+B$, then showing that it is also a 
pseudo mode for $\tilde{A}^2+B+S$, and finally 
applying the averaging lemma. 

If $S^*X$ denotes the unit cosphere bundle of
$X$, one has a diagram
\begin{equation}\label{diagrama}
\begin{array}{ccc}
\iota:S^*X &\hookrightarrow & T^*X\setminus\{ 0\}\\
\pi\downarrow\  \ & &\\
\calO & &\\
\end{array}
\end{equation}
and $\pi^*\wt V$ is the symbol of $B$ restricted to $S^*M$.
Let us assume the hypotheses of part (1) of the theorem, taking 
without loss of generality $\sigma_B(m) =\mu = 0$.
Thus, if $m\in\gamma$
\newcommand{\ave}{\text{\tiny ave}}
\begin{equation}\label{condicion}
\PB{\Re \wt V}{\Im \wt V}(m) < 0.
\end{equation}
Let $\calR\subset T^*X\setminus\{0\}$ be the ray through $m$.   $\calR$ is a conic
isotropic submanifold of $T^*X$.  Associated to any such manifold are spaces of distributions
$J^k(X, \calR)$, the Hermite distributions of Boutet de Monvel and Guillemin \cite{BdMG}.

\begin{lemma}
	For every $\ell$ there exists $\psi \in J^\ell(X, \calR)$  (with non-zero principal symbol)
	such that $B(\psi)\in C^\infty(X)$.
\end{lemma}
\begin{proof}
	The construction is symbolic.  The symbol of an element $\psi \in J^\ell(X, \calR)$
	at $m\in\calR$ is an object in the space
	\[
	\sigma_m\in \bigwedge^{1/2}\calR_m\otimes\calS(\Sigma_m),
	\]
	where:
	\begin{enumerate}
		\item $\calR_m = T_m\calR$ and $\wedge^{1/2}$ stands for half forms on it,
		\item $\Sigma_m = \calR_m^\circ/\calR_m$ is the symplectic normal space to $\calR_m$
		(here $\calR_m^\circ$ is the symplectic orthogonal to $\calR_m$ inside $T_mS^*X$).
		\item $\calS(\Sigma_m)$ is the space of smooth vectors for the representation of the
		Heisenberg group of $\Sigma_m$ in a quantization of $\Sigma_m$.
	\end{enumerate}
	Given that $\calR$ is one-dimensional and the symbol is homogeneous, the
	half form portion of the symbol is trivial:  It is $dr^{\ell/2}$.
	The interesting part of the symbol is the vector in $\calS(\Sigma_m)$.
	
	We first note that $\Sigma_m$ is naturally isomorphic with $T_\gamma\calO$.  The
	isomorphism arises from the fact that $\omega(\partial_r,\Xi) = 1$, where
	$\Xi$ is the Hamilton field of $\norm{\xi}$ and $\partial_r$ is the radial field on $T^*X$.
	Therefore the intersection $\calR_m^\circ\cap T_m(S^*X)$ is both naturally isomorphic
	to $\Sigma_m$ and to $T_\gamma\calO$.
	
	The argument is now identical to that in \cite{BU03}, \S 3.
	To find the symbol  of $\psi$, we solve the equation $B(\psi) =0$ mod $C^\infty$
	using the symbol calculus
	of Hermite distributions.  For any $\psi\in J^\ell(X,\calR)$ the symbol of $B(\psi)$ as an element
	in $J^\ell(X, \calR)$ is zero, because $\sigma_B(m) = 0$.  
	Therefore $B(\psi)\in J^{\ell-1/2}(X,\calR)$.  To find the symbol of $B(\psi)$ as an element
	of this space
	we need to use the first transport equation for Hermite distributions
	(Theorem 10.2 in \cite{BdMG}).
	   The symbol of
	$B(\psi)$ is $\calL(\sigma_m)$, where $\calL: \calS(\Sigma_m)\to \calS(\Sigma_m)$
	is the operator in the infinitesimal representation of the Heisenberg group of $\Sigma_m$
	associated to the Hamilton field of $\sigma_B$ at $m$ (which is tangent to $\calR$).  
	By Lemma 3.1 in \cite{BU03}, the 
	condition (\ref{condicion}) on the Poisson bracket implies that $\calL$ is surjective with
	non-trivial kernel.  The symbol of $\psi$ is chosen in the kernel of $\calL$.  If 
	$\psi^{(1)}\in J^\ell(X,\calR)$ has such a symbol, then $B(\psi^{(1)})\in J^{\ell-1}(X,\calR)$.  
	We then modify $\psi^{(1)}$ to $\psi^{(2)} = \psi^{(1)} + \chi^{(1)}$ where 
	$\chi^{(1)}\in J^{\ell-1/2}(X,\calR)$
	is chosen so that $B(\psi^{(2)})\in J^{\ell-3/2}(X,\calR)$.  This will be the case if
	the symbol, $\tau$, of $\chi^{(1)}$ is such that
	$\calL(\tau)$ cancels the symbol of $B(\psi^{(1)})$.  Such a symbol $\tau$ exists because
	the operator $\calL$ is surjective.  Continuing in this fashion indefinitely (and doing a
	Borel-type summation) gives the result.
\end{proof}

Let $\psi$ be as in the lemma and $\psi = \sum_{k=1}^\infty\psi_k$ be its decomposition into sums of eigenfunctions corresponding to the eigenvalue clusters.  Since
$[\Delta, B] = 0$, it follows that
\[
B(\psi) = \sum_{k=1}^\infty B(\psi_k)
\]
is the cluster decomposition of $B(\psi)$, and therefore
\[
\norm{B(\psi_k)} = O(k^{-\infty})
\]
since $B(\psi)$ is smooth.
Since $\psi\in J^\ell(X, \calR)$, proceeding similarly as in \cite{BU03}, one can prove that
the sequence of norms $\{\psi_k\}$ has an asymptotic expansion as $k\to\infty$ in decreasing
powers of $k$ with non-trivial leading term.
%\[
%\norm{\psi_k} \sim \cdots
%\]
%
It follows that $\{\psi_k\}$ is a pseudo mode for $\Delta +B$ with pseudo-eigenvalues
$\lambda_k = \Lambda_k$. 
%\red{DS: end unmodified section.}

At this point, we should have a pseudo mode for $\wt A^2+B$ with $\psi_k$ an eigenfunction of $\wt A^2$ for each $k$, with eigenvalue precisely $\Lambda_k$.

\medskip

Of course, the averaging lemma related $\Delta+M_V$ to $\wt A^2+B+S$, not to $\wt A^2+B$. However:
\begin{lemma}$\frac{||S\psi_k||}{||\psi_k||}=\mathcal O(k^{-\infty})$.
\end{lemma}
\begin{proof}  For any $N\in\mathbb N$, the operator $S\wt A^N$ is smoothing (since $S$ is)
and hence is bounded on $L^2$ by some constant $C_N$. Then $||S\wt A^N\psi_{k}||\leq C_N||\psi_{k}||$. On the other hand, $||S\wt A^N\psi_{k}||=||S\Lambda_k^{N/2}\psi_{k}||=\Lambda_k^{N/2}||S\psi_{k}||$. We conclude that $||S\psi_k||\leq C_N\Lambda_k^{-N/2}||\psi_k||$; however, $\Lambda_k$ is a quadratic polynomial in $k$ with positive leading coefficient. Since $N$ is arbitrary, this completes the proof.
\end{proof}

\begin{corollary} The sequence $\{\psi_k\}$ is also a pseudo mode for $\wt A^2+B+S$. \end{corollary}
\begin{proof} We may as well assume $\psi_k$ are $L^2$-normalized; the proof then follows immediately from the previous lemma.
\end{proof}

\medskip
We now finish the proof of the first part of Theorem \ref{Main}.
Let us define $\varphi_k := e^{-F}\psi_k$ where $\{\psi_k\}$ is a pseudo mode as in the 
previous corollary such that, without loss of generality, $\norm{\psi_k}\equiv1$.  Then
\[
\norm{(\Delta + M_V - \Lambda_k I)\varphi_k} = \norm{e^{-F}(\wt A^2+B+S-\Lambda_kI)(\psi_k)};
\]
since $e^{-F}$ is a bounded operator and
and $\{\psi_k\}$ 
is a pseudomode for $\tilde A^2 + B + S$, this is $\mathcal O(k^{-\infty})$.
On the other hand, since $F$ is of order $(-1)$,
\[
(e^{-F})^* e^{-F} = I +T
\]
where $T$ is an operator of order $(-1)$.  Therefore
\[
\norm{\varphi_k}^2 = \norm{\psi_k}^2 + \inner{T\psi_k}{\psi_k}.
\]
Since the second term on the right-hand side tends to zero as $k\to\infty$,
$\{\varphi_k\}$ is a pseudo mode for $\Delta+M_V$ with pseudo-eigenvalues $\lambda_k = \Lambda_k$.  
(Recall that we assumed without loss of generality that $\mu=0$.)
This proves the first part of Theorem \ref{Main}.

\subsection{Proof of part 2}

Let us assume that $\{\Lambda_k + \mu\}$, where $\mu = (\wt{V} + q_0)(\gamma)$, is in the asymptotic pseudospectrum, 
By definition, there exists a pseudo-mode $\{\psi_k\in E_k\}$, a sequence that satisfies 
$\norm{\psi_k} = 1$ for all $k$ and
\[
\norm{(M_V^\ave + Q_0 - \mu)(\psi_k)} = O(k^{-\infty}).
\]
Let $\psi$ be the distribution $\psi = \sum_{k=1}^\infty \psi_k$.  This is clearly non-smooth, while
\begin{equation}\label{smooth}
(M_V^\ave + Q_0 - \mu I)(\psi) \in C^\infty(X)
\end{equation}
by the Sobolev embedding theorem.

Assume now that
\[
\PB{\Re\wt{V}+q_0}{\Im V} (\gamma) > 0.
\]
By Theorem 27.1.11 of \cite{Ho}, the operator $M_V^\ave + Q_0 - \mu I$ is microlocally subelliptic on the cone
over the geodesic $\gamma \subset S^*X$, with loss of $1/2$ derivatives.  Given (\ref{smooth}), it follows that
the wave-front set of $\psi$ must be disjoint from this cone.

\section{Proof of Theorem \ref{NumRangeThm}}

Recall our notation: $X$ is a Zoll manifold, and 
let $A$ be the operator appearing in (\ref{conclusionSchrodinger}).
$A$ is a first-order
pseudo-differential operator with symbol $\sigma_A(x,\xi)=\norm{\xi}_x$ and 
spectrum the eigenvalues $k=0,1,\ldots$.  Let 
$L^2(X) = \oplus_{k=0}^\infty E_k$ be the decomposition of
$L^2(X)$ into eigenspaces of $A$.

Referring to (\ref{conclusionSchrodinger}) note first that, for each $k$, 
\[
\calR_k = \Bigl\{ \frac{\inner{(Q_0+M_V)\psi}{\psi}}{\inner{\psi}{\psi}}
\;;\;\psi\in E_k\setminus\{0\};
\Bigr\}
\]
that is, $\calR_k$ is the numerical range of the operator $Q_0+M_V+R_{-1}$ restricted to the finite-dimensional subspace $E_k$. It is clear that $\calR_k$ is closed for each $k$ and, by the Toeplitz-Hausdorff theorem, $\calR_k$ is convex. An elementary argument shows that the limit $\calR_{\infty}$ is also closed and convex.

\subsection{Existence of modes with microsupport on geodesics}

The purpose of this section is to prove:  
\begin{proposition} Let $X$ be either a standard $S^n$ or a Zoll surface, and let
	$\gamma \subset S^*X$ be a geodesic.  Then there exist sequences $\{u_k\}$
	of functions such that:
	\begin{enumerate} 
		\item $\forall k\ u_k\in E_k$ and $\norm{u_k}=1$, 
		\item the semi-classical wave front set of the sequence is
		equal to $\gamma$, and
		\item For all pseudo-differential operators $Q$ of order zero on $X$
		one has
		\begin{equation}\label{matCoeffProperty}
				\inner{Q(u_k)}{u_k} = \frac 1{2\pi}\int_\gamma \sigma_Q\, ds + O(1/\sqrt{k})
		\end{equation}
		where $\sigma_Q$ is the principal symbol of $Q$.
	\end{enumerate}
\end{proposition}
\begin{proof}
	First suppose $X$ is a sphere.  Since the construction is SO$(n+1)$ equivariant, without
	loss of generality $\gamma$ corresponds to the intersection
	$\gamma_0$  of $S^n\subset\bbR^{n+1}$
	with the $x_1x_2$ plane, where $(x_1, \ldots, x_{n+1})$ are the standard coordinates in 
	$\bbR^{n+1}$.  Then one can take
	\begin{equation}\label{canTake}
	u_k(x) = a_k\, (x_1+ix_2)^k
	\end{equation}
	where $a_k\in\bbR$ is chosen so that the $L^2$ norm of $u_k$ is equal to one.  It is known
	that these functions have the required properties (see Proposition 7.6 in \cite{Gu79}).
	
	Now assume that $X$ is a Zoll surface, which we normalize so that its geodesics have length $2\pi$.  
	Then it is known that there exists an invertible
	Fourier integral operator $U: L^2(X)\to L^2(S^2)$ such that
	\begin{equation}\label{}
	U\,\Delta\, U^{-1} = \Delta_0 + R,
	\end{equation}
	where $\Delta_0$ is the standard Laplacian on $S^2$ and $R$ is a pseudo-differential operator
	of order zero.  (See \cite{Wei74}.)  
	The FIO $U$ is associated to a homogeneous canonical transformation 
	$T: T^*S^2\setminus\{0\}\to T^*X\setminus\{0\}$ 
	that intertwines the geodesic flows.  Applying once again the averaging
	method (more precisely Lemma 1 in \cite{Gu81}) to $\Delta_0+R$, 
	one can assume without loss of generality that 
	$ [\Delta_0, R]=0$.
	(The result cited says that any operator of the form $\Delta_0 +R$ with
	$R$ self-adjoint pseudo-differential of order zero can be conjugated, this time by a unitary 
	pseudo-differential operator, to an operator of the same form but now such that 
	$ [\Delta_0, R]=0$.)  For each $k$, the operator $U$ maps the space of spherical harmonics
	of degree $k$ onto $E_k$.
	
	Pre-composing $U$ with a rotation, we can also assume that $T$ maps
	the geodesic $\gamma_0$ to $\gamma$.
	We then define $u_k$ to be the result of applying the operator
	$U$ to the right-hand side of (\ref{canTake}).
	Since the desired properties are equivariant with respect to actions of unitary FIOs, 
	we are done.

\end{proof}

This existence result immediately implies: 
\begin{corollary}
The image of $\wt V+q_0$, and therefore its convex hull, 
are contained in $\calR_\infty$.  
\end{corollary}
\begin{proof}
	Given $\gamma\in\calO$, let $\{u_k\}$ be as in the previous Proposition.   
	By (\ref{matCoeffProperty}), the limit of the matrix coefficients 
	\[
	\inner{(Q_0+M_V)u_k}{u_k}\in\calR_k
	\]
	as $k\to\infty$ is precisely $(q_0+\wt{V})(\gamma)$.
\end{proof}

\subsection{The converse} We now prove that $\calR_\infty$ is contained in the convex hull
of the image of $\wt{V} + q_0$.
(The proof is modeled on the proof of \cite[Prop. 5.2]{BU03}.)
To show this, we will show that for all lines in $\mathbb C$ such that one of the half
planes cut out by the line contains the range of $\wt V+q_0$, every element of $\calR_{\infty}$ is contained in the same half plane. This will show that $\calR_{\infty}$ is contained in, hence equal to, the convex hull in question. 

We will use the following Lemma (closely related to the sharp G\aa rding inequality):
\begin{lemma}
	Let $Q$ be a zeroth order, self-adjoint pseudo-differential operator on $S^n$. Assume that
	its principal symbol $q\in C^\infty(T^*S^n\setminus \{0\})$ is non-negative:  $q\geq 0$.  Let
	$\{ \psi_k\in E_k\}$
	be a sequence of spherical harmonics such that $\norm{\psi_k}=1$ for all $k$.  Assume furthermore that
	\[
	\lim_{k\to\infty} \inner{Q(\psi_k)}{\psi_k} = \ell\in\bbR.
	\]
	Then $\ell\geq 0$.
\end{lemma}
\begin{proof}
	By the sharp G\aa rding inequality (see \cite{GSj}, Exercise 4.9), there exists a pseudo-differential operator $B$
	of order $-1$ such that $\forall u\in L^2(S^n)$
	\[
	\inner{Q(u)}{u} \geq -\inner{B(u)}{u}.
	\]
	As before, let $A$ be the operator which is equal to multiplication by $k$ when restricted to $E_k$.  Recall that
	$A$ is first-order pseudo-differential, so that $BA$ is of order zero and therefore bounded in $L^2$.  Therefore
	$\exists C>0$ such that, with $\psi_k$ as in the hypotheses of the Lemma, $\forall k$ 
	\[
	k\left|\inner{B(\psi_k)}{\psi_k} \right| = \left| \inner{BA(\psi_k)}{\psi_k}\right| \leq C.
	\]
	Therefore $\forall k$
	\[
	\inner{Q(\psi_k)}{\psi_k} \geq -\frac{C}{k}.
	\]
	Taking limits as $k\to\infty$ yields the desired result.
	\end{proof}

To finish the proof of the converse, consider a line 
\[
\left\{u+iv\in\bbC\;;\; au+bv=c,\ a,b,c\in\bbR\right\}
\] in $\mathbb C$, and assume that the range of $\wt V+q_0$ is contained in the region 
$au+bv\geq c$. Let $u_0+iv_0\in\calR_{\infty}$.  We need to show that $au_0+bv_0\geq c$. By definition of $\calR_\infty$, 
there exists a sequence %$u_k+iv_k\in\calR_k$ and 
$\left\{\psi_k\in E_k\right\}$ such that $\forall k\ \norm{\psi_k}=1$ and such that the sequence of complex numbers
\[
u_k+iv_k:=\inner{(Q_0+M_V^\ave)\psi_k}{\psi_k} 
\]
converges to $u_0+iv_0$ as $k$ goes to infinity. Consider now the operator
\begin{equation}\label{}
Q = a\left(Q_0 + M_{\Re V}^\ave\right) + b M_{\Im V}^\ave - cI.
\end{equation}
This is a self-adjoint pseudo-differential operator of order zero with non-negative symbol, and
\begin{equation}\label{}
\inner{Q(\psi_k)}{\psi_k} = au_k + bv_k - c\  \xrightarrow[k\to\infty]{}\  au_0 + b v_0 -c.
\end{equation}
By the Lemma this limit is non-negative, and the proof is complete.

\section{Examples on $S^2$}

\subsection{Preliminaries on the averaging operator}
We present here preliminary results on the operator $V\mapsto\wt V$ on the two sphere
$S^2\subset\bbR^3$
(see also the appendix in \cite{Gu76}).
The space, $\calO$, of oriented geodesics on $S^2$
can be identified with a copy of $S^2$.  Indeed an oriented great circle on $S^2$ is the intersection
of $S^2$ with an oriented plane through the origin:  
\[
\gamma = S^2\cap \pi_\gamma.
\]
The identification $\calO\cong S^2$ is via the map $\gamma\mapsto $ the unit normal
vector to $\pi_\gamma$ defining the orientation.   
Clearly this identification is equivariant with respect to the action of 
SO$(3)$.  We can therefore regard the averaging operator as an operator on $C^\infty(S^2)$:
\begin{equation}\label{aveMap}
\begin{array}{ccc}
S^2 & \xrightarrow{\small{ave}} & S^2\\
V & \mapsto & \wt V
\end{array},
\end{equation}
where $\quad \wt V(\gamma) = \frac{1}{2\pi}\int_\gamma V\, ds$.  This operator is obviously
SO$(3)$ equivariant, and therefore it maps each space of spherical harmonics, $E_k$, into itself.
Furthermore, by Schur's lemma, it is a constant $c_k$ times the identity on $E_k$.   

To proceed further, 
let $(x,y,z)$ denote the ambient ($\bbR^3$)
coordinates and let $\zeta = x+iy$.  Note that, for each $k$, 
\begin{equation}\label{}
\zeta^k\in E_k
\end{equation}
(we will abuse the notation and denote the restrictions of $x,y,z$ to the sphere by the same letters).
Indeed $(x+iy)^k$ is a homogeneous polynomial of degree $k$ and it is harmonic, by the Cauchy-Riemann 
equations.   Following the identification $\calO\cong S^2$ one finds that
\[
c_k = \wt{\zeta^k}(1,0,0) = \frac{1}{2\pi}\int_0^{2\pi}(i\cos(t))^{k}\,dt =
\begin{cases}
0 & \text{if } k\ \text{is odd}\\
(-1)^l\,\frac{(2l-1)(2l-3)\cdots 3\cdot 1}{(2l)(2l-2)\cdots 4\cdot 2}& \text{if } k=2l 
\end{cases}
\]
Using Wallis' formula for $\pi$ it follows that 
$c_{2l} \sim (-1)^l (l\pi)^{-1/2}$ .

\subsection{Analytic functions}  Our first example is actually a class of examples, namely
potentials of the form
\[
V(x,y,z) = f(\zeta),\quad f(\zeta) = \sum_{l=1}^\infty a_l\zeta^{2l}
\]
where the series is assumed to have a radius of convergence $>1$.
In \S 5  of \cite{GuU83} it was shown that the eigenvalues of $\Delta +M_V$ are exactly the
same as those of $\Delta$ (note that there is no constant term in the series), 
the intuition being that the average of the operator $M_V$, restricted
to ${E_k}$, is nilpotent for each $k$.

To see what Theorem \ref{Main} says about the asymptotic pseudospectrum in this example, 
we need to investigate the Poisson bracket condition on $\wt V$.
Note that, by the previous discussion,
\[
\wt V = \wt{f}(\zeta) = \sum_{l=1}^\infty a_l\,c_{2l}\zeta^{2l}
\]
which shows that $\wt V$ is of the same form as $V$.

\begin{lemma}
	Let $\wt V = F+iG$ with $F$ and $G$ real-valued.  Then
	\begin{equation}\label{}
	\PB{F}{G} = z\left(F_x^2 + F_y^2\right).
	\end{equation}
\end{lemma}
\begin{proof}
	Denote by $(\theta,z)$ toric coordinates on  $\calO\cong S^2$, so that the symplectic form
	is $dz\wedge d\theta$.  Then
	\[
	\PB{F}{G} = F_zG_\theta -F_\theta G_z = 
	\]
	\[
	=( x_zF_x+y_zF_y)(x_\theta G_x + y_\theta G_y) - 
	( x_\theta F_x+y_\theta F_y)(x_z G_x + y_z G_y).
	\]
	Using the Cauchy-Riemann equations and simplifying one obtains
	\[
	\PB{F}{G} =( x_zF_x+y_zF_y)(-x_\theta F_y + y_\theta F_x) - 
	( x_\theta F_x+y_\theta F_y)(-x_z F_y + y_z F_x)
	\]
	\[
	=(x_z y_\theta - x_\theta y_z)\,(F_x^2 +F_y^2).
	\]
	But $\PB{x}{y}=z$.
\end{proof}

As the pre-image under $\wt V$ of any complex number contains points
whose $z$ coordinates differ by a sign, we get:
\begin{corollary}
	For each
	$\displaystyle{
		\mu\in\{\wt{f}(\zeta)\;;\; |\zeta|<1\ \text{and}\ \wt{f}'(\zeta)\not=0 \}}$
	the sequence $\{\Lambda_k + \mu\}$ is
	in the asymptotic pseudospectrum of $\Delta +M_V$.
\end{corollary}

\subsection{Quadratic examples}

In this section we take $V$ of the form
$V= (ax+iy)^2$, where $a$ is a real constant not equal to either zero or one.
If $\wt\Delta$ denotes the (negative) Laplacian on $\bbR^3$, then
$\wt\Delta V = 2(a^2-1).$  It follows that 
\[
h= (ax+iy)^2 +\frac 13 (1-a^2)(x^2+y^2+z^2)
\]
is a harmonic, homogeneous polynomial on $\bbR^3$.  Therefore, the
decomposition of $V$ into spherical harmonics is
\[
V|_{S^2}  = h|_{S^2} -  \frac 13 (1-a^2).
\]
This allows us to compute $\wt V$, which is, up to constants, basically $V$ itself:
\[
\wt V = -\frac 12 h-  \frac 13 (1-a^2)\quad\Rightarrow\quad \wt V = -\frac 12 V
-\frac 12(1-a^2).
\]
Therefore
\[
\PB{\Re \wt V}{\Im \wt V} = \frac 14\PB{a^2x^2-y^2}{2axy} = 
\frac a2\left( \PB{a^2x^2}{xy}- \PB{y^2}{xy}\right) =
\]
\[
=\frac a2\left(  2a^2x^2z  - 2y^2(-z)\right) = az\left(a^2x^2 + y^2\right).
\]
\begin{corollary}
	For each
	$\displaystyle{
		\mu\in\{\wt{V}(\zeta)\;;\; 0<|\zeta|<1 \}}$
	the sequence $\{\Lambda_k + \mu\}$ 
	is in the asymptotic pseudospectrum of $\Delta + M_{(ax+iy)^2}$.
\end{corollary}
%The region $\{\wt{V}(\zeta)\;;\; 0<|\zeta|<1 \}$ is a punctured ellipse.  On the other hand, is
%there anything that can be
%said about the eigenvalue clusters in this case?  THE FOLLOWING NEEDS
%TO BE FIXED.  We claim that the eigenvalues of 
%\[
%\Delta +  \left(M_{(ax+iy)^2}\right)^{\text{ave}},\quad 
%\left(M_{(ax+iy)^2}\right)^{\text{ave}} = \frac{1}{2\pi}\int_0^{2\pi}
%\]
%are well-approximated, for
%large cluster index $k$, by the numbers $\Lambda_k+\lambda_{k,j}^2$, where  
%$\{\lambda_{k,j}\;;\;3 -k\leq j\leq k\}$
%are the eigenvalues of the Toeplitz operator on $\calO$ with 
%multiplier $H = ax+iy$ on the $k$-th tensor power of  the quantum line bundle $\calL\to\calO$.  
%The latter eigenvalues are exactly computable using representation theory of SO$(3)$, and can be
%seen to equal
%\begin{equation}\label{}
%\lambda_{j,k} = j\cdot \frac{\sqrt{a^2-1}}{k},\quad -k\leq j\leq k.
%\end{equation} 
%(This is because the combination of Pauli matrices $a\sigma_x + i\sigma_y$
%can be conjugated to $-\sqrt{a^2-1}\sigma_z$.)

%\subsection{Pseudo-differential examples}
%A quick review of the proof of Theorem \ref{Main} shows that it immediately
%generalizes to the case when $M_V$ is replaced by a classical zeroth-order
%pseudo-differential operator, $Q$.  The only change that is needed is to replace
%$\wt V$ by the push-forward to $\calO$ of the average, with respect to geodesic flow,
%of the principal symbol of $Q$ restricted to the unit tangent bundle of $S^2$.
%
%In this vein, it is interesting to take
%\[
%Q = \Delta^{-1/2}\Xi,
%\]
%where $\Xi$ is a (complex) vector field on $S^2$.

\bibliographystyle{amsalpha}

\bibliographystyle{amsplain}

\begin{thebibliography}{A}
	
	\bibitem{bes78} Besse, A. ``Manifolds All of Whose Geodesics Are Closed." Ergeb. Math., Vol. 93, Springer-Verlag, New York, 1978.
	
	\bibitem{BU03} Borthwick, D. and Uribe, A. On the pseudospectra of Berezin-Toeplitz operators. Meth. Appl. Anal., Vol. 10, No. 1, 031-066, 2003.
	
	\bibitem{BdMG} Boutet de Monvel, L. and Guillemin, V.  {\em Spectral theory of Toeplitz
		operators}. Annals of Mathematics Studies {\bf 99}. Princeton U. Press, 1981.
	
	\bibitem{cdv79} Colin de Verdiere, Y. Sur le spectre des op\'erateurs elliptiques \`a bicharact\'eristiques toute p\'eriodiques. Comm. Math. Helv. 54, 508-522, 1979.
	
	\bibitem{GK69} Gohberg, I. C.; Krein, M. G.
	{\em Introduction to the theory of linear nonselfadjoint operators.}
	Translated from the Russian by A. Feinstein. Translations of Mathematical Monographs, 
	Vol. 18 American Mathematical Society, Providence, R.I. 1969.
	
	\bibitem{Gu76} Guillemin, V.
	The Radon transform on Zoll surfaces
	Adv. in Math. 22 (1976), no. 3, 85-119.
	
	\bibitem{Gu79} Guillemin,V.  Some microlocal aspects of analysis on compact symmetric spaces, 
	in {\em Seminar on Micro-Local Analysis}, Annals of Mathematics Studies {\bf 93}, Princeton Univ. Press,Princeton, N.J., 1979.
	
	\bibitem{Gu81} Guillemin, V.
	Band asymptotics in two dimensions. 
	Adv. in Math. 42 (1981), no. 3, 248-282.
	
	\bibitem{GuU83} Guillemin, V. and Uribe, A. Spectral properties of complex potentials.
	Trans AMS {\bf 279} no. 2 (1983), 759-771.
	
	\bibitem{GSj}A. Grigis and J. Sj\"prstrand, {\em Microlocal Analysis for Differential Operators}. Cambridge University Press, Cambridge, 1994.
	
	\bibitem{Ho} L. H\"ormander, {\em The Analysis of Linear Partial Differential Operators} Vol. 4, Springer Verlag.
	
	\bibitem{Ka} Kato, T. {\em Perturbation Theory for Linear Operators}, Springer Verlag,
	corrected printing of the second edition, 1995.
	
	\bibitem{P-S04} Pravda-Starov, K.  A general result about the pseudo-spectrum of 
	Schr\"odinger operators.  Proc. R. Soc. Lond. A ({\bf 460}) (2004), 471-477.
	
	\bibitem{AU85} Uribe, A.  Band invariants and closed trajectories on $S^n$. 
	Advances in Math {\bf 58} (1985), 285-299.
	
	
	\bibitem{Wei74}   Weinstein, A.  Fourier integral operators, quantization, and the spectra of Riemannian manifolds. In {\em Géométrie symplectique et physique mathématique} (Colloq. Internat. CNRS, No. 237, Aix-en-Provence, 1974), pp. 289 - 298. \'Editions Centre Nat. Recherche Sci., Paris, 1975.
	
	\bibitem{wei77} Weinstein, A. Asymptotics of eigenvalue clusters for the Laplacian plus a potential. Duke Math. J. 44, 883-892, 1977.

	
	\bibitem{zel97} Zelditch, S. Fine structure of Zoll spectra. Jour. Func. Anal. 143, 415-460, 1997.
	
	\bibitem{Zw} Zworski, M. {\em Semiclassical Analysis}.  Graduate Studies in Mathematics {\bf 138}, Amer. Math. Soc. 2012.
\end{thebibliography}
%    Insert the bibliography data here.

\end{document}